\documentclass{article}
\pdfoutput=1
\usepackage[T1]{fontenc}
\usepackage[cp1250]{inputenc}
\usepackage{amsthm}
\usepackage{amsmath}
\usepackage{amsfonts}
\usepackage{multicol}
\usepackage{epsfig}
\usepackage{verbatim}
\usepackage[affil-it]{authblk}
\usepackage[authoryear]{natbib}
\usepackage{subcaption}
\usepackage{amssymb}
\usepackage{pdfpages}

\usepackage{hyperref}
\hypersetup{%  http://www.tug.org/applications/hyperref/
	bookmarksnumbered,
	pdfstartview={FitH},
	linkcolor=black, %A
	citecolor=black, %B
	colorlinks=true, %C... A,B,C set black hyperlinks
}%

\newtheorem{theorem}{Theorem}

\theoremstyle{definition}

\theoremstyle{plain}

\newtheorem{proposition}{Proposition}

\newcommand\R{\mathbb R}

\newcommand\diag{\mathrm{diag}}

\newcommand\Var{\mathrm{Var}}
\newcommand\Cov{\mathrm{Cov}}

\title{$E$- and $R$-optimality of block designs for treatment-control comparisons}
\author{Samuel Rosa}
\affil{Faculty of Mathematics, Physics and Informatics, Comenius University, Bratislava, Slovakia}
\date{\today} %e.g. April 23, 2015}

\begin{document}
	
\maketitle

\begin{abstract}
	We study optimal block designs for comparing a set of test treatments with a control treatment. We provide the class of all $E$-optimal approximate block designs characterized by simple linear constraints. Employing this characterization, we obtain a class of $E$-optimal exact designs for treatment-control comparisons for unequal block sizes. In the studied model, we justify the use of $E$-optimality by providing a statistical interpretation for all $E$-optimal approximate designs and for the known classes of $E$-optimal exact designs. Moreover, we consider the $R$-optimality criterion, which minimizes the volume of the rectangular confidence region based on the Bonferroni confidence intervals. We show that all approximate $A$-optimal designs and a large class of $A$-optimal exact designs for treatment-control comparisons are also $R$-optimal. This further reinforces the observation that $A$-optimal designs perform well even for rectangular confidence regions.
\end{abstract}

\section{Introduction}

Consider a blocking experiment for comparing a set of test treatments with a control. As noted by \cite{HedayatEA}, such an experimental objective arises, for instance, in screening experiments or in experiments in which it is desired to assess the relative performance of new test treatments with respect to a standard treatment. Comparisons with a control are also quite natural for medical studies involving placebo (e.g., see \cite{Senn97}, \cite{RosaHarman17}).

Formally, we have
$$Y_j = \mu + \tau_{i(j)} + \theta_{k(j)} + \varepsilon_j, \quad j=1,\ldots,n, $$
where $\mu$ is the overall mean, $\tau_i$ is the effect of the $i$-th treatment ($0\leq i \leq v$), $\theta_{k}$ is the effect of the $k$-th block ($1 \leq k \leq d$), and the random errors $\varepsilon_1,\ldots,\varepsilon_n$ are uncorrelated, with zero mean and variance $\sigma^2<\infty$. Treatment $0$ denotes the control, and test treatments are numbered $1,\ldots,v$. By $\tau$, we denote the vector of treatment effects and by $\theta$ the vector of block effects.
The objective of estimating the comparisons of the test treatments with the control $\tau_i-\tau_0$ ($1 \leq i \leq v$) can be expressed as the estimation of $Q^T\tau$, where $Q:=(-1_v, I_v)^T$. Here $1_v$ denotes the column vector of ones in $\R^v$ and $I_v$ denotes the $v \times v$ identity matrix; moreover by $0_v$, we denote the column vector of zeros.
\bigskip

The most popular optimality criteria for treatment-control comparisons are $A$- and $MV$-optimality; for a survey, see \cite{HedayatEA} or \cite{Majumdar}. An $A$-optimal design minimizes the average variance of the (least squares) estimators of the $(\tau_i - \tau_0)$'s and an $MV$-optimal design minimizes the maximum variance of the estimators of the $(\tau_i-\tau_0)$'s. 

%Formally, an $A$-optimal design minimizes the trace of $N^{-1}(\xi)$ and an $MV$-optimal minimizes the maximum diagonal element of $N^{-1}(\xi)$.

In this paper, we first provide approximate and exact designs that are optimal with respect to the $E$-optimality criterion, which minimizes the maximum variance for the linear combinations $\sum_{i>0} x_i\tau_i - (\sum_{i>0} x_i)\tau_0$ over all normalized $x\in \R^v$. 
The $E$-optimality received little attention for treatment-control experiments, with some exceptions: see \cite{MajumdarNotz}, \cite{Notz}, \cite{MorganWang}, \cite{RosaHarman17}. The lower interest resulted from the fact that this criterion was considered to be lacking a natural statistical interpretation (see \cite{MajumdarNotz}, \cite{HedayatEAR}). Indeed, the aforementioned minimax interpretation of $E$-optimality does not seem to be very compelling. 

Naturally, the justification for $E$-optimality for treatment-control comparisons received significant attention in the mentioned papers. \cite{MorganWang} examined the minimax interpretation more closely, using the so-called weighted variances. \cite{Notz} studied $E$-optimal designs in the presence of two-way heterogeneity. He showed that designs in a particular class of $E$-optimal row-column designs minimize the variance of the estimator of $\sum_{i>0} \tau_i/v - \tau_0$. \cite{RosaHarman17} studied optimal approximate block designs under the presence of some special (dose-escalation) constraints when the number of blocks is $v$ or $v+1$. They provided interpretation of the obtained $E$-optimal approximate designs analogous to that of \cite{Notz}, and provided an interpretation for the obtained designs by means of variances and covariances for the $(\tau_i-\tau_0)$'s. 

We extend the results of \cite{RosaHarman17} by providing \emph{all} $E$-optimal approximate designs for comparisons with a control for any number of blocks and treatments, and without the dose-escalation design constraints. We apply this characterization to obtain $E$-optimal exact block designs  for blocks of unequal sizes. Additionally, we extend the interpretation based on the variances and covariances for the contrasts of interest to all obtained $E$-optimal approximate designs, as well as to entire classes of $E$-optimal exact designs.
\bigskip

Besides the various optimality criteria, another approach for assessing block designs for the comparisons with a control is based on simultaneous confidence intervals; see, e.g., \cite{BechhoferTamhane}, \cite{Majumdar96b}, \cite{BortnickEA}. Under the assumption of normal errors, simultaneous confidence intervals for the $(\tau_i-\tau_0)$'s based on the multivariate $t$-distribution can be calculated. Then, one seeks designs that maximize the coverage probability of these confidence intervals, subject to a limit on the width of the intervals. Thus, this approach deals with rectangular confidence regions, instead of the confidence ellipsoids, which are in some sense minimized in the case of the standard optimality criteria like $A$- and $E$-optimality. However, the problem of maximizing the coverage probability of the simultaneous confidence intervals seems to be difficult to solve.

It is known that the $A$-optimal treatment proportions are close to the optimal proportions for these rectangular confidence regions (see \cite{BechhoferTamhane83}). In the second part of this paper, we examine the actual optimality of $A$-optimal block designs (rather than merely $A$-optimal treatment proportions) with respect to a different criterion based on rectangular confidence regions -- the so called $R$-optimality.
The criterion of $R$-optimality was proposed by \cite{Dette} in general experimental settings, and it minimizes the volume of the rectangular confidence region based on Bonferroni confidence intervals for the particular contrasts of interest. For the treatment-control comparisons, $R$-optimality allows for considering rectangular confidence regions, and at the same time circumventing the technical difficulties that are tied to the simultaneous confidence intervals based on the multivariate $t$-distribution.

\subsection{Experimental designs}

An \emph{exact design} $\xi$ determines in each block the numbers of trials that are performed with the various treatments. Thus, $\xi$ can be expressed as a function $\xi: \{0,\ldots,v\}\times\{1,\ldots,d\} \to \{0,1,2,\ldots,n\}$ such that $\sum_{i,k} \xi(i,k) = n$. The value $\xi(i,k)$ specifies the number of trials performed with treatment $i$ in block $k$. Suppose that the blocks $1,\ldots, d$ have pre-specified non-zero sizes $m_1, \ldots, m_d$. We denote the class of all block designs for $v+1$ treatments and $d$ blocks of sizes $m=(m_1,\ldots,m_d)^T>0_d$ by $D(v,d,m)$. If all blocks are assumed to be of the same size, say $q$, we write $D(v,d,q1_d)$.

For a given design $\xi$, let us denote the $(v+1) \times d$ design matrix $X(\xi)=(\xi(i,k))_{i,k}$, let $r(\xi)=X(\xi)1_d$ be the vector of total treatment replications and let $s(\xi)=X^T(\xi) 1_v$ be the vector of block sizes. Clearly, any $\xi \in D(v,d,m)$ must satisfy $s(\xi)=m$. Furthermore, we define $\mu_{ij}(\xi) = \sum_k \xi(i,k)\xi(j,k)$, $0\leq i,j \leq v$. For brevity, we usually omit the argument $\xi$ in $X(\xi)$, $r(\xi)$ etc.

The information matrix of an exact design $\xi$ for estimating all pairwise comparisons of treatments is $M(\xi) = \diag(r) - X\diag(s^{-1})X^T$, where $s^{-1}:=(s_1^{-1},\ldots,s_d^{-1})$. The treatment-control comparisons $Q^T\tau$ are estimable under $\xi$ if $\mathcal{C}(Q) \subseteq \mathcal{C}(M)$, where $\mathcal{C}$ denotes the column space. In such a case, we say that $\xi$ is feasible, and we have $\mathrm{rank}(M)=v$ and $\mathrm{rank}(Q^TM^-Q)=v$. The information matrix $N(\xi)=(Q^TM^-Q)^{-1}$ of a feasible design $\xi$ for estimating $Q^T\tau$ is obtained by deleting the first row and column of $M$ (see \cite{BechhoferTamhane}). Let us partition $X$ as $X^T = (z, Z^T)$, where $z$ is a $d \times 1$ vector; i.e., $Z=(\xi(i,k))_{i>0,k}$.
Then, the information matrix for comparing the test treatments with the control is \begin{equation}\label{eInfMat}
N(\xi) = \diag(r_1,\ldots,r_v) - Z\diag(s^{-1})Z^T.
\end{equation}
Note that $N(\xi)$ is proportional to the inverse of the covariance matrix of the least squares estimator of $\tau_1-\tau_0, \ldots, \tau_v-\tau_0$. A design is said to be $\Psi$-optimal if it minimizes $\Psi(N(\xi))$ for some function $\Psi$. The criteria of $A$-, $MV$-, and $E$-optimality can be defined as follows: an $A$-optimal design minimizes the trace of $N^{-1}(\xi)$, an $MV$-optimal minimizes the maximum diagonal element of $N^{-1}(\xi)$, and an $E$-optimal design minimizes the largest eigenvalue of $N^{-1}(\xi)$.
\bigskip

An \emph{approximate design} is a function $\xi:\{0,\ldots,v\}\times\{1,\ldots,d\} \to [0,1]$, such that $\sum_{i,k} \xi(i,k) = 1$. Then, the value $\xi(i,k)$ represents the proportion of all trials for treatment $i$ and block $k$, rather than the actual number of trials. Analogously to the exact case, we define the terms $X(\xi)$, $Z(\xi)$, $r(\xi)$ etc. Since the sizes of the blocks are assumed to be positive, we always have $s>0_d$. Under the usual assumption of blocks of the same size, $s$ satisfies $s=1_d/d$.

The treatment-control comparisons $Q^T\tau$ are estimable under an approximate design $\xi$ if $\mathcal{C}(Q) \subseteq \mathcal{C}(M(\xi))$. In such a case, we say that $\xi$ is feasible and the information matrix of $\xi$ for treatment-control comparisons is given as in the exact case by \eqref{eInfMat}.

We say that a design $\xi$ that satisfies $\xi(i,k) = r_i s_k$ for all $i$, $k$ is a \emph{product design} of $r$ and $s$. We denote such a design as $\xi = r \otimes s$.

%\subsection{Optimal treatment proportions}
%
%To analyze both $E$-optimal and $R$-optimal designs, we employ the notion of optimal treatment proportions.
%Consider the model without block effects
%\begin{equation}\label{eModelTreat}
%Y_j = \tau_{i(j)} + \varepsilon_j, \quad j=1,\ldots,n.
%\end{equation}
%and suppose that $r$ is an approximate design in this model. That is, $r$ is a function $r: \{0,\ldots,v\} \to [0,1]$, such that $\sum_i r(i) = 1$. Thus, $r$ can be represented by a $(v+1)\times 1$ vector, where $r_i$ determines the proportion of all trials performed with treatment $i$. The contrasts $Q^T\tau$ are estimable under $r$ if and only if $r>0$ and the information matrix of such a design for $Q^T\tau$ under \eqref{eModelTreat} is $N_\tau(r)=(Q^T \diag(r^{-1}) Q)^{-1}$, see \cite{RosaHarman16}.
%
%We say that $r^*$ are $\Psi$-optimal treatment proportions if $r^*$ minimizes $\Psi(N_\tau(r))$. In other words, such $r^*$ is a $\Psi$-optimal design in model \eqref{eModelTreat}. Moreover, we say that a block design $\xi$ attains $\Psi$-optimal treatment proportions if the treatment proportions of $\xi$ are $\Psi$-optimal, i.e., $r(\xi)=r^*$. The optimal treatment proportions specify the optimal (relative) treatment replications, which are then allocated to the particular blocks.

\section{$E$-optimality}

\subsection{Approximate designs}\label{ssEoptApprox}

A design is $E$-optimal if it minimizes $\lambda_{\max}(N^{-1}(\xi))$ or, equivalently, if it maximizes $\lambda_{\min}(N(\xi))$, where $\lambda_{\max}$ and $\lambda_{\min}$ denote the largest and the smallest eigenvalue, respectively.
In the following theorem, we provide the complete characterization of $E$-optimal approximate block designs for comparing the test treatments with a control: an approximate design $\xi^*$ is $E$-optimal for comparisons with a control if and only if (i) in each block $\xi^*$ assigns one half of the trials to the control and (ii) $\xi^*$ is equireplicated in the test treatments.

\begin{theorem}\label{tEopt}
	An approximate block design $\xi$ is $E$-optimal for treatment-control comparisons if and only if $\xi$ satisfies
	\begin{equation}\label{eEopt}
	\xi(0,k)=\frac{s_k(\xi)}{2} \quad\text{and}\quad r_1(\xi) = \ldots = r_v(\xi) = \frac{1}{2v}.
	\end{equation}
\end{theorem}

\begin{proof}
	Let $\xi$ be $E$-optimal and let $r_0^*=1/2$ and $r_i^* = 1/(2v)$ for $i>1$. From Theorems 1 and 6 of \cite{RosaHarman16} it follows that a $\Psi$-optimal block design must satisfy $r(\xi)=r^*$. Moreover, the product design $\tilde{\xi}=r^* \otimes s$ for any $s>0$, $\sum_k s_k = 1$, is $E$-optimal (cf. \cite{Schwabe}, \cite{RosaHarman16}). Then $Z(\tilde{\xi}) = 1_v s^T/(2v)$ and
	$$N(\tilde{\xi}) = \frac{1}{2v} I_v -  \frac{1}{4v^2} 1_v s^T \diag(s^{-1}) s 1_v^T = \frac{1}{2v} I_v - \frac{1}{4v^2} J_v.$$
	Therefore, the optimal smallest eigenvalue is $\lambda^*=1/(4v)$.
	
	Moreover, \eqref{eInfMat} yields
	$$\begin{aligned}
	\lambda_{\min}(N(\xi))
	&= \min_{x^Tx = 1} x^T N(\xi)x \leq \frac{1}{v}1_{v}^T N(\xi) 1_{v} \\
	&= \frac{1}{v} \big( \sum_{i>0} r_i - \sum_{k=1}^d \frac{1}{s_k} (\sum_{i>0} \xi(i,k))^2 \big) 
	= \frac{1}{v} \big( \frac{1}{2} - \sum_{k=1}^d \frac{(s_k-\xi(0,k))^2}{s_k} \big) \\
	&= \frac{1}{v} \big( \frac{1}{2} - \sum_{k=1}^d s_k + 2\sum_{k=1}^d \xi(1,k) - \sum_{k=1}^d\frac{\xi^2(0,k)}{s_k} \big) 
	= \frac{1}{2v} - \frac{1}{v}\sum_{k=1}^d\frac{\xi^2(0,k)}{s_k} ,
	\end{aligned}$$
	using the facts that $\sum_{i>0} r_i = 1/2$, $\sum_{i>0} \xi(i,k) = s_k - \xi(0,k)$ and $\sum_k s_k = 1$. Since $\sum_k \xi(0,k) = 1/2$, for any $s_1, \ldots, s_d$ the following holds
	$$\sum_{k=1}^d \frac{\xi^2(1,k)}{s_k} \geq \sum_{k=1}^d \frac{(s_k/2)^2}{s_k} = \frac{1}{4},$$
	where the inequality is attained as equality if and only if $\xi(1,k)=s_k/2$ for all $k=1,\ldots,d$. Hence, 
	\begin{equation}\label{eLambdaMin}
	\lambda_{\min}(N(\xi)) \leq \frac{1}{v}(\frac{1}{2}-\frac{1}{4}) = \frac{1}{4v}=\lambda_{\min}^*.
	\end{equation}
	Because $\xi$ is $E$-optimal, the inequality is attained as equality, and thus $\xi(1,k)=s_k/2$ for all $k=1,\ldots,d$.
	
	For the converse part, let $\xi$ satisfy \eqref{eEopt}. Then, $\xi$ is connected (see \cite{EcclestonHedayat}) and therefore feasible. Moreover, $Z^T 1_v= s/2$, $Z1_d = (r_1, \ldots, r_v)^T = 1_v/(2v)$ and
	$$N1_v = \frac{1}{2v}1_v - \frac{1}{2} Z \diag(s^{-1}) s = \frac{1}{2v}1_v - \frac{1}{4v}1_v = \frac{1}{4v}1_v.$$
	That is, the optimal eigenvalue $\lambda^*=1/(4v)$ is an eigenvalue of $N$ corresponding to the eigenvector $1_v$. Therefore, it suffices to prove that $\lambda^*$ is the smallest eigenvalue of $N$.
	
	Let $N = (n_{ij})_{i,j}$. We note that $n_{ij}\leq 0$ for $i\neq j$.
	Using an argument analogous to that in the proof of Theorem 3.1 by \cite{MajumdarNotz}, let $x$ be an eigenvector of $N(\xi)$. Let us denote the eigenvalue that corresponds to $x$ as $\lambda$. By multiplying $x$ by an appropriate constant, we obtain $\max_j |x_j| = 1$. Thus, $x_j \leq 1$ for all $1\leq j \leq v$. Let $i$ be the index that satisfies $|x_i|=1$. Then, by multiplying $x$ by $\pm1$, we obtain $x_i=1$. Now, we can write
	$$(Nx)_i = n_{ii} x_i + \sum_{j \neq i} n_{ij} x_j \geq n_{ii} + \sum_{j \neq i} n_{ij} = (N1_v)_i,$$
	where the inequality follows from $n_{ij} \leq 0$ for $j\neq i$, and $x_j \leq 1$ for $1\leq j \leq v$. Because $(Nx)_i = \lambda x_i = \lambda$ and $(N1_v)_i = \lambda^*$, we have $\lambda^* \leq \lambda$ for any eigenvalue $\lambda$.
\end{proof}

\subsection{Interpretation of $E$-optimality}

It is well known that an $E$-optimal design minimizes the maximum variance of  linear functions $x^T Q^T\tau$ of $Q^T\tau$, over all normalized $x \in \R^v$. That is, an $E$-optimal design for comparisons with a control minimizes the maximum variance for $x^TQ^T\tau = \sum_{i>0} x_i\tau_i - (\sum_{i>0} x_i)\tau_0$ over all $\lVert x \rVert = 1$, as mentioned in the Introduction.

We say that a design $\xi$ is $c$-optimal, where $c \in \R^{v+1}$, if it minimizes the variance of the least squares estimator of $c^T\tau$. Therefore, $\xi$ is $c$-optimal if it minimizes $c^T M^{-}(\xi) c$.
The relationship between $E$-optimality and $c$-optimality for approximate designs (see \cite{PukStudden}) allows for a statistically meaningful interpretation of $E$-optimality, analogous to \cite{Notz}, \cite{RosaHarman17}. Moreover, this relationship provides a statistical justification for \emph{all} approximate designs that are $E$-optimal for treatment-control comparisons.

\begin{theorem}\label{tEoptcopt}
	Let $\xi$ be an $E$-optimal approximate block design for comparisons with a control. Then
	\begin{itemize}
		\item[(i)] $\xi$ is optimal for $\tilde{c}^T\tau$, where $\tilde{c}^T=(-1,1_v^T/v)$;
		\item[(ii)]  $\xi$ minimizes the variance of the estimator of $(\sum_{i>0} \tau_i)/v - \tau_0$;
		\item[(iii)] %$\xi$ minimizes the sum of variances and of absolute values of covariances for $\tau_1-\tau_0, \ldots, \tau_v-\tau_0$. That is,
		$\xi$ minimizes
		\begin{equation}\label{eVarCovSum}
		\sum_{i=1}^v\Var(\widehat{\tau_i-\tau_0}) + \sum_{1\leq i\neq j \leq v} \lvert\Cov(\widehat{\tau_i-\tau_0}, \widehat{\tau_j-\tau_0})\rvert
		\end{equation}
	\end{itemize}
\end{theorem}

\begin{proof}
	Using Theorem 2.2 and Corollary 2.3 of \cite{PukStudden}, we obtain that if a design $\xi$ is $E$-optimal for $Q^T\tau$, $\lambda_{\min}(N(\xi))$ has multiplicity 1 and $h$ is the corresponding eigenvector, then any $E$-optimal design is optimal for $h^TQ^T\tau$.
	
	Let $\xi=r^* \otimes s$, where $r^*=(1/2,1_v/(2v))^T$, be an $E$-optimal product design. Then, $N(\xi) = \frac{1}{2v} I_v - \frac{1}{4v^2} J_v$ as shown in the proof of Theorem \ref{tEopt}. The smallest eigenvalue of $N(\xi)$ has multiplicity 1 and the corresponding eigenvector is $1_v$. Therefore, any $E$-optimal design is optimal for $1_v^T Q^T \tau = (-1,1_v^T/v)\tau$. Parts (i) and (ii) are clearly equivalent.
	
	The optimality for $Q1_v$ means that any $E$-optimal design $\xi$ minimizes $1_v^T Q^T M^-(\xi) Q 1_v = 1_v^T N^{-1}(\xi) 1_v$, which is the sum of all elements of the inverse of the information matrix (which in turn is proportional to the covariance matrix). Therefore, $\xi$ minimizes the sum of variances and covariances for the treatment-control comparisons. Because for any feasible $\xi$, the non-diagonal elements of $N(\xi)$ are non-positive and every real eigenvalue of $N(\xi)$ is positive, Theorem 6.2.3 of \cite{BermanPlemmons} yields that every element of $N^{-1}(\xi)$ is non-negative. Thus, the minimization of the sum of the covariances is equivalent to the minimization of the sum of the absolute values of the covariances.
\end{proof}

Theorem \ref{tEoptcopt}(ii) states that $E$-optimal designs also minimize the variance of the estimator for comparing the average test treatment effect with the control effect. Thus, $E$-optimal designs guarantee that the overall effect of the test treatments relative to control is estimated with the greatest precision. The key interpretation of $E$-optimality is part (iii), which states that any $E$-optimal design for treatment-control comparisons minimizes the sum of variances and of absolute values of covariances for $\tau_1-\tau_0, \ldots, \tau_v-\tau_0$.
This interpretation is similar to that of \cite{RosaHarman17} for dose-escalation designs, and is analogous to that of $A$- or $MV$-optimality, but based on both the variances and the covariances of the estimators.
In the discussion on \cite{HedayatEA}, multiple authors (\cite{BechhoferTamhane88}, \cite{GiovagnoliVerdinelli}) criticized the usually used criteria of $A$- and $MV$-optimality on the basis that they deal only with the variances of the estimators of $\tau_i-\tau_0$, disregarding the covariances between the estimators. The provided interpretation suggests that especially when one is also interested in the covariances of the estimators, $E$-optimality is a reasonable choice.

\subsection{Exact designs}

For the usual case of blocks of the same size $D(v,d,q1_d)$, \cite{MajumdarNotz} provide a wide class of $E$-optimal designs (Theorem 3.1 therein), which we state in the following proposition.

\begin{proposition}\label{pEopt}
	If there exists an exact design $\xi \in D(v,d,q1_d)$ that satisfies (i) $\xi(0,k) = q/2$ if $q$ is even, or $\xi(0,k) \in \{ \lfloor q/2 \rfloor, \lfloor q/2 \rfloor +1\}$ if $q$ is odd $(0\leq k \leq d)$, and (ii) $\mu_{01}(\xi) = \ldots = \mu_{0v}(\xi)$, then $\xi$ is $E$-optimal for comparisons with a control in $D(v,d,q1_d)$.
\end{proposition}

Because the conditions for $E$-optimality of approximate designs in Section \ref{ssEoptApprox} are not very stringent, $E$-optimal exact designs can often be directly obtained from the approximate ones. Indeed, in the following theorem, we provide such $E$-optimal exact designs for many combinations of $v$, $d$, $m$. Moreover, these exact designs retain the $\tilde{c}$-optimality property and they minimize the sum of the variances and covariances for $(\tau_i-\tau_0)$'s.

\begin{theorem}\label{tEoptExact}
	If there exists an exact design $\xi \in D(v,d,m)$ that satisfies $\xi(0,k)=m_k/2$, $1 \leq k \leq d$, and $\xi$ is equireplicated in the test treatments, then $\xi$ is $E$-optimal for comparisons with a control in $D(v,d,m)$. Moreover, $\xi$ is optimal for $(\sum_{i>0} \tau_i)/v - \tau_0$ and it minimizes \eqref{eVarCovSum}.
\end{theorem}

\begin{proof}
	The approximate version $\xi/n$ of $\xi$ is in fact an $E$-optimal approximate design, because it satisfies the conditions of Theorem \ref{tEopt}. Then, $\xi$ is clearly an $E$-optimal exact design. Moreover, Theorem \ref{tEoptcopt} yields the optimality for $\tilde{c}^T\tau$ and for \eqref{eVarCovSum}.
\end{proof}

For blocks of equal size,  Theorem \ref{tEoptExact} becomes a special case of Proposition \ref{pEopt}. However, Theorem \ref{tEoptExact} also covers the case of blocks of unequal sizes. For $v=2$ and for blocks of unequal sizes $m=(2,2,4)^T$, a simple $E$-optimal block design $\xi$ for treatment-control comparisons, which also minimizes \eqref{eVarCovSum} is given by
$$X(\xi)=\begin{pmatrix}
 1  &  1 &   2  \\
 1 &  0  &  1 \\
 0  &  1 &   1  
\end{pmatrix}.$$
\bigskip

We provide the interpretation of $E$-optimality for comparisons with a control for approximate designs through \eqref{eVarCovSum} also for all $E$-optimal exact designs given by \cite{MajumdarNotz}, which is, to our knowledge, the widest known class of $E$-optimal block designs for comparisons with a control.

\begin{theorem}\label{EoptVarCovSum}
	All $E$-optimal designs given by Proposition \ref{pEopt} are also optimal for $(\sum_{i>0} \tau_i)/v - \tau_0$ and minimize the sum of variances and the absolute values of covariances \eqref{eVarCovSum}.
\end{theorem}

\begin{proof}
	The technique of the proof is analogous to that of Theorem 3.2 of \cite{Notz}.
	Let $\xi^*$ be given by Proposition \ref{pEopt}, let $\xi \in D(v,d,q1_d)$ and let us denote 
	$$\bar{N}(\xi) = \frac{1}{v!} \sum_{\pi-\text{perm.}} P_\pi N( \xi)P_\pi^T ,$$
	where the sum is over all permutations $\pi$ of the $v$ test treatments and $P_\pi$ is the permutation matrix for $\pi$. Then, 
	$$\begin{aligned}\bar{N}(\xi) 
	= &\frac{1}{v}\left( \sum_{i=1}^v r_i - \frac{1}{m}\sum_{i=1}^v\sum_{k=1}^d \xi^2(i,k) + \frac{1}{q(v-1)} \sum_{1 \leq i \neq j \leq v} \mu_{ij} \right) I_v \\
	&- \frac{1}{qv(v-1)} \left( \sum_{1 \leq i \neq j \leq v} \mu_{ij}  \right) J_v,\end{aligned}$$
	and its smallest eigenvalue
	$\lambda_{\min}(\bar{N}) = \sum_{i=1}^v \mu_{0i}/(qv)$ corresponds to the eigenvector $1_v$, see Lemma 2.1 by \cite{MajumdarNotz}. Convexity of $N\mapsto N^{-1}$ yields $\bar{N}^{-1}(\xi) \preceq (1/v!) \sum_{\pi} P_\pi N^{-1}(\xi)P_\pi^T$. Therefore, using the fact that $P_\pi 1_v = 1_v$,
	$$ 1_v^T N^{-1}(\xi) 1_v = \frac{1}{v!} \sum_{\pi} 1_v^T P_\pi N^{-1}(\xi)P_\pi^T 1_v \geq 1_v^T \bar{N}^{-1}(\xi) 1_v = \frac{qv^2}{\sum_i\mu_{0i}(\xi)}.$$
	In the proof of Theorem 3.1, \cite{MajumdarNotz} show that $\xi^*$ maximizes $\sum_{i=1}^v \mu_{0i}(\xi)/(qv)$ and that $1_v$ is an eigenvector corresponding to the smallest eigenvalue of $N(\xi^*)$, which is $\mu_{01}(\xi^*)/q$. Hence,
	$$ 1_v^T N^{-1}(\xi) 1_v \geq \frac{qv^2}{\sum_i\mu_{0i}(\xi)} \geq \frac{qv^2}{\sum_i\mu_{0i}(\xi^*)} = \frac{qv}{\mu_{01}(\xi^*)} = 1_v^T N^{-1}(\xi^*) 1_v,$$ 
	because $\mu_{01}(\xi^*) = \ldots = \mu_{0v}(\xi^*)$. The matrix $N^{-1}(\xi^*)$ is non-negative by an argument analogous to the approximate case.
\end{proof}

One example of $E$-optimal design $\xi$ for $v=4$ and $m=4\cdot1_4$ provided by \cite{MajumdarNotz}, which Theorem \ref{EoptVarCovSum} shows to be optimal with respect to the sum of variances and covariances \eqref{eVarCovSum}, is given by
$$
X(\xi)=\begin{pmatrix}
2 & 2 & 2 & 2 \\
1 & 0 & 0 & 1 \\
1 & 1 & 0 & 0 \\
0 & 1 & 1 & 0 \\
0 & 0 & 0 & 1
\end{pmatrix}.
$$

Note that all $E$-optimal designs given by \cite{MorganWang} minimize \eqref{eVarCovSum} too, because they form a subset of the designs given by Proposition \ref{pEopt}: the designs given by \cite{MorganWang} must satisfy, among other conditions, that $\xi(0,k)=\lfloor q/2 \rfloor$ for all $k$, and they are equally replicated and binary in the test treatments. It follows that $\mu_{0j}(\xi) = \sum_k \xi(0,k) \xi(j,k) = \lfloor q/2 \rfloor \sum_k \xi(j,k)$, which does not depend on $j$ because $\xi$ is equireplicated in the test treatments. Clearly, such designs satisfy the conditions of Proposition \ref{pEopt}.

\section{$R$-optimality}

\subsection{Approximate designs}

An $R$-optimal design minimizes
$$\Psi_R(N(\xi)) = \prod_{i=1}^s (N^{-1}(\xi))_{ii}.$$
Equivalently, an $R$-optimal design maximizes the concave function
$$\Phi_R(N(\xi)) = \Big(\prod_{i=1}^v (N^{-1}(\xi))_{ii}\Big)^{-1/v}.$$
Note that $R$-optimality is permutationally invariant; i.e. $\Phi_R(PNP^T)=\Phi_R(N)$ for any $v \times v$ permutation matrix $P$.

Due to the symmetric nature of the studied problem, the thoroughly studied $A$-optimal designs tend to be $R$-optimal. In fact, we show that any $A$-optimal approximate block design for treatment-control comparisons is also $R$-optimal. An approximate block design $\xi$ is $A$-optimal for treatment-control comparisons if and only if $\xi^*=r^*\otimes s$, where
\begin{equation}\label{eAopt}
r^*_0 = \frac{1}{\sqrt{v}+1} \quad\text{and}\quad r^*_i =\frac{1}{\sqrt{v}(\sqrt{v}+1)} \text{ for } i>0,
\end{equation}
(see \cite{RosaHarman16}, \cite{GiovagnoliWynn}).

\begin{theorem}
	Let $\xi^*$ be an $A$-optimal approximate block design for comparisons with a control, i.e., $\xi^*=r^* \otimes s$, where $r^*$ is given by \eqref{eAopt}. Then, $\xi$ is also $R$-optimal.
\end{theorem}

\begin{proof}
	Let $\xi_0$ be a feasible block design and let $\pi$ be a permutation of the $v$ test treatments. We define $P_\pi\xi_0$ as the design obtained from $\xi_0$ by the permutation $\pi$ of the test treatment labels. Then, $N(P_\pi \xi_0) = P_\pi N(\xi_0) P_\pi^T$. Let us denote
	$$\bar{\xi}_0 = \frac{1}{v!} \sum_{\pi-\text{perm.}}  N( P_\pi \xi_0)  = \frac{1}{v!} \sum_{\pi-\text{perm.}} P_\pi N(\xi_0) P_\pi^T,$$
	where the sum is over all permutations $\pi$ of the test treatments. Then, from the permutation invariance and the concavity of $\Phi_R$, we have $\Phi_R(\bar{\xi}_0) \leq \Phi_R(\xi_0)$. For any design $\xi$ equireplicated in the test treatments, its information matrix $N$ is completely symmetric, and therefore the diagonal elements of $N^{-1}$ coincide. Let us denote an arbitrary diagonal element of $N^{-1}$ for such $N$ as $g(N)$. Then, $\xi^*$ minimizes $\Psi_A(\xi) = vg(N(\xi))$ among all designs equireplicated in test treatments; thus, it maximizes $(g(N(\xi)))^{-1} = \Phi_R(\xi)$. It follows that $\Phi_R(\xi^*) \geq \Phi_R(\bar{\xi}_0)$, which together with $\Phi_R(\bar{\xi}_0) \geq \Phi_R(\xi_0)$ yields the $R$-optimality of $\xi^*$.
\end{proof}

\subsection{Exact designs}

\cite{MajumdarNotz} provided a class of BTIB designs that are $A$-optimal for comparisons with a control among all block designs (Theorem 2.2 therein).
Analogously to the approximate case, we show that any such $A$-optimal design is also $R$-optimal; thus extending the relationship between $A$- and $R$-optimality to exact designs. 

\begin{theorem}\label{tRopt}
	Let $\xi^* \in D(v,d,q1_d)$ be an $A$-optimal BTIB design for comparisons with a control given by Theorem 2.2 of \cite{MajumdarNotz}:
	(i) $\xi^*$ is binary in test treatments, (ii) $\sum_k \xi(0,k) = R$, (iii) $\xi(0,k)=\lfloor R/d \rfloor$ or $\xi(0,k)=\lfloor R/d \rfloor + 1$ for $1 \leq k \leq d$, where $R$ is the value of the integer $r$ $(0\leq r \leq \lfloor dq/2\rfloor)$ that minimizes
	$$ g(r;d,q,v) = \frac{v}{r-h(r;d)/q} + \frac{(v-1)^2}{d(q-1)-r(q-1)/q - v(r-h(r;d)/q)},$$
	where
	$$h(r;d)=\lfloor r/d \rfloor^2 \left(d+d\lfloor r/d \rfloor - r\right) + \left(r-d\lfloor r/d \rfloor \right) \left(r/d + 1 \right)^2.$$
	Then, $\xi^*$ is also $R$-optimal for comparisons with a control in $D(v,d,q1_d)$.
\end{theorem}

\begin{proof}
	Let $\xi$ be a block design and let $\bar{N}(\xi)$ be defined as in the proof of Theorem \ref{EoptVarCovSum}. Then, $\Phi_R(\bar{N}(\xi)) \geq \Phi_R(N(\xi))$.
	Because $N(\xi^*)$ is completely symmetric, we have $\bar{N}(\xi^*) = N(\xi^*)$. From the proof of Theorem 2.1 of \cite{MajumdarNotz} it follows that $\xi^*$ minimizes $\Psi_A(\bar{N}(\xi))$ over all block designs $\xi$. For a completely symmetric matrix $N$, the diagonal elements of $N^{-1}$ coincide; let us denote an arbitrary diagonal element of $N^{-1}$ for such $N$ as $t(N)$. Then, $\xi^*$ minimizes $\Psi_A(\bar{N}(\xi)) = st(\bar{N}(\xi))$; thus, it maximizes $(t(\bar{N}(\xi)))^{-1} = \Phi_R(\bar{N}(\xi))$ over all block designs $\xi$. Therefore, $\Phi_R(N(\xi^*))=\Phi_R(\bar{N}(\xi^*))\geq \Phi_R(\bar{N}(\xi)) \geq \Phi_R(N(\xi))$.
\end{proof}

Multiple examples and entire families of designs satisfying Theorem \ref{tRopt} can be found in \cite{HedayatEA}.

\section{Discussion}

We provided the class of all $E$-optimal approximate block designs for comparisons with a control, which can be described by the simple linear constraints \eqref{eEopt}. For a strictly convex criterion, the only optimal block designs are product designs with optimal treatment proportions (see \cite{RosaHarman16}). However, since $E$-optimality lacks strict convexity, the class of $E$-optimal designs is richer and thus allows for an easier construction of efficient (or even optimal) exact designs by the rounding methods -- therefore increasing the usefulness of the approximate theory. We demonstrated this by easily constructing a class of $E$-optimal exact designs, which extends the known results on $E$-optimality to the case of unequal block sizes.

$E$-optimality is often considered to be inappropriate for comparisons with a control, because of its lack of a natural statistical interpretation. However, we formulated such an interpretation in the examined experimental settings. Namely, we proved that the $E$-optimal approximate designs minimize the sum of all variances and the absolute values of covariances for the comparisons with a control \eqref{eVarCovSum}; compare with $A$-optimality, which seeks to minimize the sum of variances for the $(\tau_i-\tau_0)$'s.  Moreover, we showed that the $E$-optimal exact designs provided in this paper as well as all $E$-optimal exact designs provided by \cite{MajumdarNotz} also have this statistical interpretation.
The provided interpretation suggests that $E$-optimality is a reasonable criterion for treatment-control comparisons. In fact, $E$-optimality has the advantage of taking the covariances for the ($\tau_i-\tau_0$)'s into consideration, unlike the $A$- and $MV$-optimality, which were criticized for their `blindness' to covariances (see \cite{GiovagnoliVerdinelli}, \cite{BechhoferTamhane88}).

We also provided an additional justification for using $A$-optimal block designs for comparisons with a control by showing that wide classes of these designs are also $R$-optimal. Specifically, we prove that any $A$-optimal approximate block design for treatment-control comparisons is $R$-optimal. We then extend this relationship to exact designs by proving that any $A$-optimal exact design given by \cite{MajumdarNotz} is also $R$-optimal. Therefore, not only are $A$-optimal proportions efficient for the rectangular confidence regions based on the multivariate $t$-distribution (cf. \cite{BechhoferTamhane83}), but also the $A$-optimal block designs are often actually in some sense optimal for the rectangular Bonferroni confidence regions. Therefore, even if one is mostly interested in the rectangular confidence regions, one may prefer to use the $A$-optimal block designs, which perform `well enough' for such regions, because they are easier to obtain than the optimal designs for the confidence intervals based on the multivariate $t$-distributions.

\bibliographystyle{plainnat}
\bibliography{C:/BibTeX/rosa.bib}

\end{document}